\documentclass[11pt,a4paper]{article}
\usepackage[latin1]{inputenc}
\usepackage{indentfirst}
\usepackage[dvips]{graphicx}
\usepackage{amsfonts}
\usepackage{amssymb}
\usepackage{amsmath}
\usepackage{amsthm}
\usepackage{geometry}
\usepackage{pstricks,pst-node}
\usepackage{placeins}

\newcommand{\R}{{\mathbb R}}

\usepackage{float}

\newtheorem{thm}{Theorem}[section]
\newtheorem{prop}[thm]{Proposition}
\newtheorem{lemma}[thm]{Lemma}

\theoremstyle{definition}

\newtheorem{defi}[thm]{Definition}

\numberwithin{equation}{section}

\begin{document}

\title{Towards Dynamic PET Reconstruction under Flow Conditions: Parameter Identification in a PDE Model}

\author{Louise Reips\thanks{Universidade Federal de Santa Catarina, Campus Blumenau, CEP 89065-300, Blumenau-SC, Brasil} \and Martin Burger\thanks{Institute for Computational and Applied Mathematics and Cells in Motion Cluster of Excellence, Westf\"alische Wilhelms-Universit\"at (WWU) M\"unster, Einsteinstr. 62, 48149 M\"unster, Germany} \and  Ralf Engbers$^\dagger$}
\date{}
\maketitle

\begin{abstract}
The aim of this paper is to discuss potential advances in PET kinetic models and direct reconstruction of kinetic parameters. As a prominent example we focus on a typical task in perfusion imaging and derive a system of transport-reaction-diffusion equations, which is able to include macroscopic flow properties in addition to the usual exchange between arteries, veins, and tissues. 

For this system we propose an inverse problem of estimating all relevant parameters from PET data.  We interpret the parameter identification as a nonlinear inverse problem, for which we formulate and analyze variational regularization approaches. 
For the numerical solution we employ gradient-based methods and appropriate splitting methods, which are used to investigate some test cases.
\end{abstract}

\section{Introduction}

Positron Emission Tomography (PET) is a nuclear medical imaging technique, used to visualize and quantify metabolic and physiological processes in the human body. The use of $^{18}$F-FDG PET is a  widely established method to quantify metabolism e.g. in tumours, but other investigations based on different tracers are still far from daily clinical use, although they offer great opportunities due to their inherent dynamical structure. For certain tracers yielding appropriate quality data at a reasonable time scale, kinetic modelling is an established technique based on estimating coefficients in ODE models from mean values of the acticity in large regions of interest (cf. \cite{Wernick}). For data of lower quality, e.g. in tracers like $H_2^{15}O$ with low half-life (fast decay), novel approaches based on the direct reconstruction of parameters from PET data instead of an intermediate image reconstruction step has evolved as a promising tool recently (cf. \cite{Benningkoestwuebbelingschaefburg,BenningHeinsBurg}), at least for preclinical investigations. The advances in direct reconstruction, with the time evolution of activity distributions constrained by a system of ODEs in subregions, raises the hope to obtain more detailed local pictures of physiological parameters like perfusion. An obvious benefit of imaging spatial variance in perfusion is that local defects can be detected and even quantified. However, with the localization the simple modelling by ODEs becomes questionable, in particular transport effects need to be taken into account and are not averaged out as in large regions of interest. 

Here we present a model-based approach to overcome those difficulties. We derive a set of partial differential equations able to represent the kinetic behavior of $H_2^{15}O$ PET tracers during cardiac perfusion. As in kinetic ODE models, we rely on a homogenized formulation, i.e. we do not resolve the single arteries, veins or even capillaries in the tissue region, and take into account the exchange of materials between artery, tissue and vein. The main difference to those models is the addition of transport and diffusion terms to take into account the local flow behaviour. This model predicts the tracer activity if the reaction rates, velocities, and diffusion coefficients are known, but again we are interested in the inverse problem of identifying those distributed parameters. Under a natural stationary flow assumption those parameters are only spatially dependent, but constant in time. In this way the inverse problem from dynamic data becomes hopefully overdetermined, which raises the hope to obtain decent reconstructions even for bad data statistics.

Using a model encoded in an operator $G$ mapping parameters $p$ to a time evolution of activity $u$, we can formulate the dynamic inverse problem in PET as
\begin{equation} \label{eq:inverseproblem}
	\wp(Ku(t)) = f(t), \qquad u=G(p)
\end{equation}
where $K$ is the forward operator $f(t)$ is a sequence of measured PET data on a domain $\Sigma$, and  $\wp(z)$ denotes a Poisson random variable with expectation $z$. 

The state-of-the art approach for solving problems with statistical noise models is a Bayesian formulation, in particular MAP (maximum a-posteriori probability estimation, cf. \cite{kaipio2006statistical}), which in our case yields a minimization problem of the form (cf. \cite{Shepp,Vardi,emtv})
\begin{equation} \label{variational0}
 (u,p) \in \arg \min_{(u,p) }   \int\limits_0^T \int_{\Sigma} Ku(t) - f(t) log(Ku(t)) dy~dt  + \alpha \mathcal{R}(p) 
\end{equation}
where $\mathcal{R}$ is a regularization fuctional, which in our case incorporates smoothness and a-priori information about typical values of parameters. We will in detail investigate the properties of the forward operator and the variational scheme when $G$ is the solution operator of a system of partial differential equations with linear diffusion, transport and reaction terms, where $p$ is the vector of spatially distributed parameters in such models. Moreover, we discuss appropriate schemes for the numerical solution, based on splitting of the PET forward operator and the PDE constraint, which allows to use well-established techniques for the subproblems of image reconstruction (split into different time steps) and the parameter identification problem.


The remainder of the paper is organized as follows: In Section 2 we introduce the novel PDE-based forward model and analyze its basic properties. In Section 3 we discuss the nonlinear inverse problem of estimating parameters from PET data, with particular focus on its variational regularization. In Section 4 we state the basic ingredients for our numerical solution methods, which we apply in some test cases in Section 5. 

\section{Three-Component Reaction-Diffusion Model}

In the following we introduce our macroscopic model of cardiovascular perfusion  and provide a basic mathematical analysis. 

Standard tracer kinetic modelling in PET is based on compartmental models (\cite{Carson2,Wernick}). In a PET image sequence, fixed spatial compartments are areas defined by the concentration of a radioactive tracer (called activity) that is a a temporal function. As a way of describing the interaction between these compartments one associates a constant capable to represent the velocity of absorption, diffusion of the radioactive tracer used during the PET scan. Thus data concerning the rate at which radioactive tracer is metabolized in the region of interest can be associated with temporal dynamics of the tracer in each compartment {\cite{Carson2}}. This standard approach yields an estimation problem for a finite number of parameters in a system of ODEs (or often just a nonlinear fitting problem since the ODE system can be solved explicitely). 

The kinetics of $^{18} F$ -fluorodeoxyglucose (FDG) and $H_2^{15}O$ are typical examples model\-led by compartmental schemes (cf.  \cite{Eriksson, Phelps01}). 
$^{13}N$-Ammonia \cite{Castelani01, Fiechtera01, Siegrist01} and $H_2^{15}O$ \cite{Ahn, Iida01, Katoh, LuedermannS,SchaefersSpinks} are the usual tracers used to estimate regional myocardial blood perfusion.
In {\cite{Muzik}} a two-compartmental model and in {\cite{Kuhle, Krivokapich}} a three-compartmental model are applied to the analysis of myocardial PET images. Direct reconstruction of distributed parameters in compartmental models from PET data are discussed e.g. in \cite{Kamasak,BenningDA,Benningkoestwuebbelingschaefburg,BenningHeinsBurg}.

We now derive our spatially distributed model for the kinetic behavior of $H_2^{15}O$ PET tracers during cardiac perfusion. 
Let $\Omega \subset \R^3$ be a bounded domain representing the relevant region on which the image and parameters are to be reconstructed and let $t \in [0,T]$. We think of $x \in \Omega$ as a macroscopic variable homogenizing microscopic flow patterns. The following processes are modelled in perfusion:
\begin{itemize}
\item Activity is transported to the tissue region in arteries. 

\item Inside the tissue, small capillaries are transporting blood, which is usually referred to as perfusion.

\item Activity is transported out of the tissue region in veins.

\item The tracer in any region is subject to radioactive decay with rate $k_0$.

\end{itemize}
For our model it is hence natural to describe the tracer activity by three parts, namely those in arteries, veins, and in tissue (capillaries). Hence, we have
\begin{equation} \label{eq:udefinition}
	u(x,t)=  C_\mathcal{A}(x,t) + C_\mathcal{V}(x,t) + C_\mathcal{T}(x,t).
\end{equation}
To each of the concentrations we associate a local velocity $V_{\mathcal{A}}, V_{\mathcal{V}}, V_{\mathcal{T}}$ respectively, and a local diffusion coefficient $ D_{\mathcal{A}}, D_{\mathcal{V}}, D_{\mathcal{T}}$. Those can be thought of as homogenized quantities at the relevant scale for our reconstruction, in particular for the microscopic capillaries. In principle we need to expect an anisotropic diffusion this way, but since we are rather interested in strong transport and only small corrections by diffusion we restrict our interest to isotropic diffusion, in particular with the goal of subsequent parameter identification. 

In addition we model the exchange of activities from arteries into tissue (as a linear reaction with rate $k_1$), from tissue into veins (with rate $k_2$) and from veins into arteries (with rate $k_3$). The latter is rather an idealization to close the model in larger domains than the heart. Thus, we obtain the system

\begin{equation}\label{partialCA}
\begin{split}
\partial_t C_\mathcal{A}(x,t) & = -k_0 C_{\mathcal{A}}(x,t)  - k_1(x)C_{\mathcal{A}}(x,t) + k_3(x) C_{\mathcal{V}}(x,t) + \underbrace{\nabla \cdot(V_{\mathcal{A}}(x) C_{\mathcal{A}}(x,t))}_{Transport} \\
& + \underbrace{\nabla \cdot(D_{\mathcal{A}}(x)\nabla C_{\mathcal{A}}(x,t))}_{Diffusion}\\
\end{split}
\end{equation}
\begin{equation}\label{partialCT}
\begin{split}
\partial_t C_{\mathcal{T}}(x,t) & = -k_0 C_{\mathcal{T}}(x,t)  + k_1(x)C_{\mathcal{A}}(x,t) - k_2(x) C_{\mathcal{T}}(x,t) + \nabla \cdot(V_{\mathcal{T}}(x) C_{\mathcal{T}}(x,t))\\
& + \nabla \cdot(D_{\mathcal{T}}(x)\nabla C_{\mathcal{T}}(x,t))\\
\end{split}
\end{equation}
\begin{equation}\label{partialCV}
\begin{split}
\partial_t C_{\mathcal{V}}(x,t)  & = -k_0 C_{\mathcal{V}}(x,t)  - k_3(x)C_{\mathcal{V}}(x,t) + k_2(x) C_{\mathcal{T}}(x,t) + \nabla \cdot(V_{\mathcal{V}}(x) C_{\mathcal{V}}(x,t)) \\
& + \nabla \cdot(D_{\mathcal{V}}(x)\nabla C_{\mathcal{V}}(x,t))\\
\end{split}
\end{equation}

This system is supplemented by boundary conditions
\begin{equation}\label{boundaryconditionsflux}
\begin{split}
(D \nabla C_{\mathcal{A/T/V}} + VC_{\mathcal{A/T/V}})\cdot \textit{n} & = j_{\mathcal{A/T/V}}^{in} \hspace{0.65cm}  \Gamma \subset \partial \Omega \\
(D \nabla C_{\mathcal{A/T/V}} + VC_{\mathcal{A/T/V}})\cdot \textit{n} & = C_{\mathcal{A/T/V}}V_{\mathcal{A/T/V}}^{out}  \hspace{0.5cm} \partial \Omega / \Gamma\\
\end{split}
\end{equation}
where $\Gamma$ denotes an inflow part of the boundary (as well as isolated parts where $j_{in} = 0$).

Since the tracer is injected into arteries, the typical initial conditions to be considered are
\begin{align}
	C_\mathcal{A}(x,0) &= C^0(x) \nonumber \\
	C_\mathcal{V}(x,0) &= 0 \label{initialconditons} \\
	C_\mathcal{T}(x,0) &= 0 \nonumber 
\end{align}
for $x \in \Omega$, with a given initial distribution $C^0$.

The differential equations \eqref{partialCA}, \eqref{partialCT}, \eqref{partialCV} can be written as an abstract evolution equation
\begin{equation}
	\partial_t {\bf C} = {\cal L} {\bf C}
\end{equation}
for 
\begin{equation}
	{\bf C} = ( C_{\mathcal{A}}, C_{\mathcal{T}}, C_{\mathcal{V}}) 
\end{equation}
with a parabolic operator ${\cal L}$, which will be a fundamental property for our analysis below.

\subsection{Existence and Uniqueness of Solutions}

We shall look for a weak solution 
\begin{equation}
	{\bf C} \in {\cal W}:= L^2(0,T;H^1(\Omega)) \cap H^1(0,T;H^{-1}(\Omega)). 
\end{equation}
for given parameters in a set
\begin{equation}\label{D0}
\mathcal{D}_0:=\{ k_i \in L^2(\Omega), V_{\mathcal{A}/\mathcal{V}/\mathcal{T}} \in L^{\infty}(\Omega), D_{\mathcal{A}/\mathcal{V}/\mathcal{T}} \in L^{\infty}(\Omega), k \geq 0, D_{\mathcal{A}/\mathcal{V}/\mathcal{T}} \geq D_0 > 0 \}.
\end{equation}
We shall below use the notation
\begin{equation} \label{equationdep}
 p(x) = (k_1(x), k_2(x), k_3(x), D_{\mathcal{T}}(x), D_{\mathcal{A}}(x), D_{\mathcal{V}}(x), V_{\mathcal{T}}(x), V_{\mathcal{A}}(x), V_{\mathcal{V}}(x))
\end{equation}
for the vector of all parameters. 

Moreover, throughout the whole paper we make the following assumptions concerning boundary and initial conditions without further notice:
\begin{itemize}
\item $C^0 \in L^2(\Omega)$, $C^0 \geq 0$ almost everywhere. 

\item $ V_{\mathcal{A/T/V}}^{out} \in L^\infty([0,T]\times \partial \Omega)$.

\item $j^{in}_{\mathcal{A}/\mathcal{V}/\mathcal{T}} \in L^2(0,T; H^{-1/2}(\Gamma))$ and $\langle j^{in}_{\mathcal{A}/\mathcal{V}/\mathcal{T}}, \varphi \rangle \geq 0$ for the trace of every nonnegative function $\varphi \in L^1(0,T;H^1(\Omega))$. 
\end{itemize} 

Standard techniques for parabolic systems (cf. \cite{Dautray}) yield the following basic results for the solution (cf. \cite{Reips} for details of the proofs):
\begin{thm} \label{existenceuniquenessthm}
For each $p \in \mathcal{D}_0$ there exists a unique solution 
${\bf C} \in {\cal W}$ of \eqref{partialCA}-\eqref{initialconditons}. Moreover, $	C_\mathcal{A}$, $	C_\mathcal{T}$, and $	C_\mathcal{V}$ are nonnegative almost everywhere in $[0,T] \times \Omega$.
\end{thm} 

As a consequence, one obtains a well-defined forward map to the activity distribution:
\begin{thm} \label{existencethm}
For every $p \in \mathcal{D}_0$ the activity distribution $u$ defined by \eqref{eq:udefinition} is uniquely defined in $L^2(0,T;H^1(\Omega)) \cap H^1(0,T;H^{-1}(\Omega))$ and is nonnegative almost everywhere.
\end{thm} 

\subsection{Continuous Dependence and Differentiability}

In order to define the nonlinear forward operator $G$, we first introduce the parameter to solution map in the form
\begin{equation}
	S: {\cal D}_0 \rightarrow {\cal W}, \quad p \mapsto {\bf C}, 
\end{equation}
which is well defined by Theorem \ref{existenceuniquenessthm}. Based on $S$ we obtain the map $G=(1,1,1) \cdot S$ as 
\begin{equation}
	G: {\cal D}_0 \rightarrow L^2(0,T;H^1(\Omega)) \cap H^1(0,T;H^{-1}(\Omega), \quad p \mapsto u.
\end{equation}

A direct estimate based on the properties of the parabolic system yields the following result (cf. \cite{Reips}):
\begin{prop} \label{lipschitzthm}
The map $S$ and consequently $G$ is locally Lipschitz-continuous. 
\end{prop}

A more lengthy calculation yields the differentiability of the forward map (cf. again \cite{Reips} for details):
\begin{thm}
The map $S$ is Frechet-differentiable on ${\cal D}_0$ with derivative
$$ (\Phi_{\mathcal{A}},\Phi_{\mathcal{T}},\Phi_{\mathcal{V}}) =S'(p){\tilde p} $$
being the unique weak solution of the system 
\begin{align}	
&\partial_t \Phi_{\mathcal{A}}(x,t) + (k_0+k_1) \Phi_{\mathcal{A}}  - k_3 \Phi_{\mathcal{V}} - \nabla \cdot(V_{\mathcal{A}} \Phi_{\mathcal{A}} - \nabla \cdot(D_{\mathcal{A}}\nabla \Phi_{\mathcal{A}})) \nonumber \\
	& \qquad = - \tilde k_1 C_{\mathcal{A}}  + \tilde k_3 C_{\mathcal{V}} + \nabla \cdot(\tilde V_{\mathcal{A}} C_{\mathcal{A}} + \nabla \cdot(\tilde D_{\mathcal{A}}\nabla C_{\mathcal{A}})) \\
	&\partial_t \Phi_{\mathcal{T}}(x,t) + (k_0+k_2) \Phi_{\mathcal{T}}  - k_1 \Phi_{\mathcal{A}} - \nabla \cdot(V_{\mathcal{T}} \Phi_{\mathcal{T}} - \nabla \cdot(D_{\mathcal{T}}\nabla \Phi_{\mathcal{T}})) \nonumber \\
	& \qquad = - \tilde k_2 C_{\mathcal{T}}  + \tilde k_1 C_{\mathcal{A}} + \nabla \cdot(\tilde V_{\mathcal{T}} C_{\mathcal{T}} + \nabla \cdot(\tilde D_{\mathcal{T}}\nabla C_{\mathcal{T}})) \\
	&\partial_t \Phi_{\mathcal{V}}(x,t) + (k_0+k_3) \Phi_{\mathcal{V}}  - k_2 \Phi_{\mathcal{T}} - \nabla \cdot(V_{\mathcal{V}} \Phi_{\mathcal{V}} - \nabla \cdot(D_{\mathcal{V}}\nabla \Phi_{\mathcal{V}})) \nonumber \\
	& \qquad = - \tilde k_3 C_{\mathcal{V}}  + \tilde k_2 C_{\mathcal{T}} + \nabla \cdot(\tilde V_{\mathcal{V}} C_{\mathcal{V}} + \nabla \cdot(\tilde D_{\mathcal{V}}\nabla C_{\mathcal{V}})) 
\end{align}
with homogeneous initial conditions and boundary conditions
\begin{align}
 &	(D_{\mathcal{A/T/V}} \nabla \Phi_{\mathcal{A/T/V}} + V_{\mathcal{A/T/V}} \Phi_{\mathcal{A/T/V}})\cdot \textit{n}  \nonumber \\
 & \qquad = - (\tilde D_{\mathcal{A/T/V}} \nabla C_{\mathcal{A/T/V}} + \tilde V_{\mathcal{A/T/V}} C_{\mathcal{A/T/V}})\cdot \textit{n} \qquad \text{in }  \Gamma \subset \partial \Omega \\
& (D_{\mathcal{A/T/V}} \nabla \Phi_{\mathcal{A/T/V}} + V_{\mathcal{A/T/V}} \Phi_{\mathcal{A/T/V}})\cdot \textit{n} \nonumber \\ & \qquad = (\Phi_{\mathcal{A/T/V}}V_{\mathcal{A/T/V}}^{out} - \tilde D_{\mathcal{A/T/V}} \nabla C_{\mathcal{A/T/V}} - \tilde V_{\mathcal{A/T/V}} C_{\mathcal{A/T/V}})\cdot \textit{n} \qquad \text{in } \partial \Omega / \Gamma.
\end{align}
\end{thm}

\section{The Inverse Problem and its Variational Regularization}

We now turn our attention to the inverse problem \eqref{eq:inverseproblem} respectively the noiseless version
\begin{equation} \label{eq:inversenonoise}
	Ku^*(t) = f^*(t), \qquad u^* = G(p^*). 
\end{equation}
The inversion can be considered as two subproblems, namely the reconstruction of an activity evolution $u$ from $f$, which is a standard image reconstruction problem at each time step $t$, and the solution of the parameter identification problem $G(p) = u$ for given $u$.

The regularization and numerical solution of the inverse problem is carried out for the full problem for the following rationale: in the case of bad data statistics stationary reconstructions at single time intervals yield results of inferior quality, and regularization at the level of the image $u$ thus yields a too strong bias, which is unnecessary since one can use a model for the time evolution instead. More natural regularizations at the level of the parameters can thus be incorporated, which better correspond to the available prior knowledge. 

The PET forward operator can be thought of as the x-ray transform (cf. \cite{Natterer}), but in real-life applications it becomes a more complicated operator including several corrections, e.g. for scattering and positron range (cf. \cite{Wernick}). Here we are not interested in these issues and the detailed structure of the operator $K$, but simply make the following assumption on $K $:
\begin{equation} \label{eq:Kcondition}
 K: L^2(\Omega) \rightarrow L^2(\Sigma) \text{ is a compact operator preserving positivity}.
\end{equation}

As mentioned above we use a Tikhonov-type regularization approach \eqref{variational0} 
with $u=G(p)$ to compute stable approximations of the solution of \eqref{eq:inversenonoise}. We will compute a minimizer of the form
\begin{equation} \label{eq:variational1}
 \hat p \in \arg \min_{p \in \mathcal{D}_p }  \mathcal{J}(p) = \int\limits_0^T \int_{\Sigma} K(G(p))(t) - f(t) log(K(G(p))(t)) dy~dt  + \alpha \mathcal{R}(p) 
\end{equation}
A subproblem we want to consider is the parameter identification problem
\begin{equation} \label{eq:variational2}
 \hat p \in \arg \min_{p \in \mathcal{D}_p }   \int\limits_0^T \int_{\Omega} \omega(x,t) (G(p)(x,t) - v(x,t))^2 dx~dt  + \alpha \mathcal{R}(p) 
\end{equation}
with given activity $v$ and a positive weight function $\omega$, which will also appear as a subproblem in our numerical approach based on operator splitting as we shall see below. 

It remains to specify the typical regularization functionals we want to employ in the reconstruction, which we shall discuss in the following and then proceed to the analysis of \eqref{eq:variational1} respectively \eqref{eq:variational2}.

\subsection{Choice of Regularization Functionals}

The main a-priori informations that can be used for choosing regularization functionals are 
\begin{itemize}

\item Typical values for some of the parameters, e.g. diffusion coefficients. 

\item Spatial smoothness of the parameters at least inside organs. 

\end{itemize}
For this sake it seems natural to construct a regularization functional as follows:
\begin{equation}
	 \mathcal{R}(p) = \int_\Omega (D(p;p^*) + E(\nabla p)) ~dx,
\end{equation}
where $D$ measures a distance from prior values $p^*$ and $E$ is a convex energy penalizing large values of the gradient of $p$. In the simplest case one can use quadratic functionals, i.e. 
\begin{equation} \label{regularization}
	 \mathcal{R}(p) = \frac{1}2 \sum_i \int_\Omega ( \beta_i (p_i - p_i^*)^2 + \gamma_i|\nabla p|^2) ~dx,
\end{equation}
with weights $\beta_i$ and $\gamma_i$, which we shall also use for our numerical tests below. 

We make the assumption that 
\begin{equation}
	\{\mathcal{R}(p) \leq C \} \cap \mathcal{D}_p \text{ is bounded in } W^{1,r}(\Omega)^{15}
\end{equation}
for some $r > \frac{6}5$, where the effective domain is given by
\begin{equation}
	\mathcal{D}_p=	\mathcal{D}_0 \cap \{ 0 < d_{min} \leq D_{\mathcal{A}/\mathcal{T}/\mathcal{V}} \leq d_{max}, |V_{\mathcal{A}/\mathcal{T}/\mathcal{V}}| \leq v_{max} \}.
\end{equation}
Note that under the above condition we also have compactness of the level sets in $L^2(\Omega)^{15}$.
Moreover, we assume that $p \mapsto \int_\Omega D(p;p^*)~dx$ is weakly lower semicontinuous with respect to weak convergence in $W^{1,r}(\Omega)^{15}$, which is true e.g. if $D$ is convex.

\subsection{Analysis of the Variational Regularization}

In the following we provide a brief analysis of the variational regularization method. Most arguments are in line with the analysis of nonlinear operator equations in Banach spaces in \cite{Schusterbuch}, but due to the Poisson data fidelity, which is not the power of a norm, they are not covered by those. 

\begin{lemma}
Under the above conditions the operator $G$ is sequentially closed on $\mathcal{D}_p$  from the weak topology of $W^{1,r}(\Omega)^{15}$ to strong convergence in $L^2(0,T;L^2(\Omega))$.
\end{lemma}
\begin{proof}
Let $p^m$ be a weakly convergent sequence with limit $p$. Then we see from Theorem \ref{existencethm} and Proposition \ref{lipschitzthm} that there exists a unique solution for any $p^m \in \mathcal{D}_p$ and the activity ${\bf C}^m=S(p^m)$ is uniformly bounded in ${\cal W}$. Hence, there exists a weakly convergent subsequence, which converges strongly in $L^2(0,T;L^2(\Omega))^3$ by the Aubin-Lions lemma (cf. \cite{showalter}). We now verify that the limit of every such subsequence is $S(p)$, which implies the convergence of the original sequence ${\bf C}^m$ due to uniqueness of the limit of subsequences. 

To do so, we have to show that weak solutions of \eqref{partialCA}-\eqref{initialconditons} with parameter $p^m$ converge to weak solutions with parameter $p$, the uniqueness result in Theorem \ref{existencethm} then finishes the argument. For brevity we carry out the analysis only for \eqref{partialCA}, whose weak formulation is given by
 
\begin{equation*}
\begin{split}
\langle \partial_t C_\mathcal{A}^m, \varphi \rangle & + \int_\Omega ( (k_0+k_1^m) C_{\mathcal{A}}^m \varphi - k_3^m C_{\mathcal{V}}^m \varphi +  V_{\mathcal{A}}^m C_{\mathcal{A}}^m  \cdot \nabla \varphi + D_{\mathcal{A}}^m \nabla C_{\mathcal{A}}^m \cdot \nabla \varphi ) ~dx\\
& = \int_\Gamma j_{\mathcal{A}}^{in} \varphi~d\sigma + 
\int_{\partial \Omega \setminus \Gamma} C_\mathcal{A}^m V_{\mathcal{A}}^{out} \varphi~d\sigma
\end{split}
\end{equation*}
for $\varphi \in H^1(\Omega)$. We first carry out the limit for $\varphi \in C^1(\Omega)$. We obtain because of weak convergence of $C_\mathcal{A}^m$ in $H^1(0,T;H^{-1}(\Omega))$ that
$$ \langle \partial_t C_\mathcal{A}^m, \varphi \rangle \rightarrow \langle \partial_t C_\mathcal{A}, \varphi \rangle. $$ 
Moreover, the compact embedding into $L^2$ yields strong $L^1$-convergence
\begin{align*}
	(k_0+k_1^m) C_{\mathcal{A}}^m  & \rightarrow (k_0+k_1) C_{\mathcal{A}} \\
	k_3^m C_{\mathcal{V}}^m  & \rightarrow k_3 C_{\mathcal{V}} \\
	V_{\mathcal{A}}^m C_{\mathcal{A}}^m  & \rightarrow V_{\mathcal{A}} C_{\mathcal{A}} 
\end{align*}
and the strong $L^2$-convergence of $D_{\mathcal{A}}^m$ together with the weak $L^2$-convergence of $\nabla C_{\mathcal{A}}^m$ yields weak $L^1$-convergence
$$  D_{\mathcal{A}}^m \nabla C_{\mathcal{A}}^m  \rightharpoonup D_{\mathcal{A}} \nabla C_{\mathcal{A}} . $$ 
Thus, we may pass to the limit in all terms on the left-hand side of the weak formulation. Since the right-hand side is a continuous affinely linear functional of $C_{\mathcal{A}}^m $ we can immediately pass to the limit. Thus, we conclude
\begin{equation*}
\begin{split}
\langle \partial_t C_\mathcal{A}, \varphi \rangle & + \int_\Omega ( (k_0+k_1) C_{\mathcal{A}} \varphi - k_3 C_{\mathcal{V}} \varphi +  V_{\mathcal{A}} C_{\mathcal{A}}  \cdot \nabla \varphi + D_{\mathcal{A}} \nabla C_{\mathcal{A}} \cdot \nabla \varphi ) ~dx \\
& = \int_\Gamma j_{\mathcal{A}}^{in} \varphi~d\sigma + 
\int_{\partial \Omega \setminus \Gamma} C_\mathcal{A} V_{\mathcal{A}}^{out} \varphi~d\sigma
\end{split}
\end{equation*}
for all $\varphi \in C^1(\Omega)$. As the final step, we use that $\mathcal{D}_p$ is closed under the above convergence of $p^m$. For $p \in \mathcal{D}_p$ it is straight-forward to see that both the left- and right-hand side can be extended to continuous linear functionals of $\varphi$ on $H^1(\Omega)$, which allows to apply a closure argument and to see that indeed ${\bf C}$ is the unique weak solution of \eqref{partialCA}-\eqref{initialconditons}.
\end{proof}

The weak closedness of the forward operator is the main ingredient to establish the lower semicontinuity of the functional $\mathcal{J}$. The remaining steps are immediate, since both the data term is convex as a function of $u$ and the regularization functional $\mathcal{R}$ is weakly lower semicontinuous due to the above assumptions:
\begin{lemma}
Under the above conditions the functional $\mathcal{J}$ is sequentially lower semicontinuous on $\mathcal{D}_p$ with respect to the weak topology in $W^{1,p}(\Omega)^{15}$.
\end{lemma}

Now we have established weak lower semicontinuity and boundedness of level sets of $\mathcal{J}$. With the Banach-Alaoglu theorem (cf. \cite{rudin}) the latter implies the following existence result: 
\begin{thm}
Let the above assumptions on $\mathcal{R}$ and $K$ hold, and let $\alpha >0$. If there exists $p \in \mathcal{D}_p$ such that $\mathcal{J}(p)<\infty$, then there exists a solution of \eqref{eq:variational2}.
\end{thm}

An analogous result can be obtained for the parameter identification subproblem, which will be useful to justify our numerical approach below:
\begin{thm}
Let the above assumptions on $\mathcal{R}$ hold, let $v \in L^2([0,T]\times \Omega)$, $\omega \in L^\infty([0,T]\times \Omega)$, and $\alpha >0$. Then there exists a solution of \eqref{eq:variational2}.
\end{thm}

The above weak compactness and lower semicontinuity are the key ingredients to verify stability of the regularization with respect to perturbations of $f$ at fixed $\alpha$, respectively convergence of the regularization as $f \rightarrow Ku^*$ and $\alpha \rightarrow 0$ with a suitable condition between $\alpha$ and the noise level with techniques developed in \cite{SawatzkyThesis}. We omit the quite technical arguments here and refer to \cite{Reips} for further details. 


\section{Numerical Solution}

In the following we discuss the numerical solution of the inverse problem \eqref{eq:inverseproblem} res-\ pectively the regularized version \eqref{eq:variational1}.
We will first discuss the overall minimization and then proceed to a more detailed discussion of the parameter identification problem and its discretization.


\subsection{Minimization Methods}

A basic paradigm for the construction of minimization methods is to avoid solutions of systems with the complicated operator $K$, whose discretization can correspond to a non-sparse matrix that one would like to avoid. Moreover, we try to avoid the solution of complicated systems for the activity distribution in space-time, but rather try to find a splitting such that stationary image reconstruction steps can be computed. For this reason we employ an operator splitting similar to previous investigations in PET and related problems (cf. \cite{emtv,BSB09a}), which can use the well-known EM-iteration as a first part, i.e. with $u_k = G(p_k)$
\begin{equation}
u_{k + \frac{1}{2}}(t) = \dfrac{u_k(t)}{K^{*}1}K^{*} \left(\dfrac{f(t)}{Ku_k(t)} \right)
\end{equation}
In the second half-step we use a backward splitting step of the form. 
\begin{equation} 
 p_{k+1} \in \arg \min_{p \in \mathcal{D}_p }   \int\limits_0^T \int_{\Omega} \omega_k(x,t) (G(p)(x,t) - u_{k + \frac{1}{2}}(x,t))^2 dx~dt  + \alpha \mathcal{R}(p) 
\end{equation}
with $\omega_k = \frac{K^{*}1}{G(p^k)}$.
The optimality condition for the latter is
\begin{equation}
	G'(p_{k+1})^*\left( \frac{G(p_{k+1}) - u_{k + \frac{1}{2}}}{G(p_k)}\right) + \alpha   \mathcal{R}'(p^{k+1}) = 0. 
\end{equation}
Inserting the first half-step yields
\begin{equation}
	G'(p_{k+1})^*\left( K^{*}1 \frac{G(p_{k+1})}{G(p_k)} - K^{*} \left(\frac{f}{KG(p_k)}   \right) \right)+ \alpha   \mathcal{R}'(p^{k+1}) = 0,
\end{equation}
which confirms the consistency with the optimality conditions. 

The major advantage of the above minimization approach is that available efficient methods can be used for the first half-step, and a more standard parameter identification problem of the form \eqref{eq:variational2} is to be solved in the second half-step, which can be performed by standard gradient-type methods, e.g. steepest descent or quasi-Newton methods. 
In our case we use a variant of a projected gradient method for the regularization functional \eqref{regularization}. The derivative of
\begin{equation}
	\mathcal{J}_k(p) =   \int\limits_0^T \int_{\Omega} \omega_k(x,t) (G(p)(x,t) - u_{k + \frac{1}{2}}(x,t))^2 dx~dt  + \alpha \mathcal{R}(p) 
\end{equation}
is given by
\begin{equation}
	\mathcal{J}_k'(p) =   G'(p)^*( \omega_k (G(p) - u_{k + \frac{1}{2}}))  + \alpha( \beta(p-p^*) - \gamma \Delta p),
\end{equation}
where we use a vectorial notation $\beta p = (\beta_i p_i)$. In particular the last term involving the Laplacian of $p$ would necessitate high damping and hence slow convergence in a standard gradient method we therefore use forward-backward splitting approach again, computing in a subiteration first a solution of (with damping parameter $\eta$)
\begin{equation}
	(\eta + \alpha \beta) q^{n+1} - \Delta q^{n+1} = \eta p^n - G'(p^n)^*( \omega_k (G(p^n) - u_{k + \frac{1}{2}})) + \alpha \beta p^*,
\end{equation}
followed by a projection step $p^{n+1} = \mathcal{P}_{\mathcal{D}_p}(q^{n+1})$.

To compute the derivative we first need to solve the forward model for given $p^n$ yielding $G(p^n)$. To subsequently compute $G'(p^n)^*( \omega_k (G(p^n) - u_{k + \frac{1}{2}}))$ we use an adjoint method (cf. \cite{Hinze01}). For detailed computations we refer to \cite{Reips}, the adjoint equations are stated in the appendix.
In order to numerically implement the iteration we need appropriate discretizations of the PDE models, which we discuss in the following.

\subsection{Discretization of the Differential Equations}\label{DiscretizacaodasEquacoesDif}

In the following we discuss the discretization of the model, for the sake of simplicity restricting ourselves to the spatially two-dimensional case, which is later also used in our computational tests. The extension to three dimensions is straight-forward.
For the time discretization of the system \eqref{partialCA}-\eqref{partialCV} we use an operator splitting approach common for reaction-diffusion systems. In addition, we use alternating directions implicit (ADI) splitting for the spatial derivatives. To provide a detailed discussion, we write the system as
\begin{equation}\label{sistemacactcv}
\partial_t C  =  \nabla ((V(x) C) + (D(x)\nabla C)) +\left(
  \begin{array}{ccc}
    -(k_0 + k_1) & k_3 & 0 \\
    0 & -(k_0 + k_3) & k_2 \\
    k_1 & 0 & -(k_0 +k_2) \\
  \end{array}
\right) C,
\end{equation}
with $$C = \left(
  \begin{array}{ccc}
    C_{\mathcal{A}} \\
    C_{\mathcal{V}} \\
    C_{\mathcal{T}} \\
  \end{array}
\right), \qquad D = \left(
  \begin{array}{ccc}
    D_{\mathcal{A}} \\
    D_{\mathcal{V}} \\
    D_{\mathcal{T}} \\
  \end{array}
\right), \qquad V = \left(
  \begin{array}{ccc}
    V_{\mathcal{A}} \\
    V_{\mathcal{V}} \\
    V_{\mathcal{T}} \\
  \end{array}
\right). $$

We discretize the system on a time grid $t_k=k\tau$ for a time step $\tau > 0$ and use the notation $C^{\tau}$ for the time discrete solution. With an ADI splitting of spatial derivatives and reaction terms we obtain
\begin{align} \label{i}
 \dfrac{C^{\tau} \left(t_k+\frac{\tau}{3} \right) - C^{\tau}(t_k)}{\tau} & = \partial_{x_1} \left(D \partial_{x_1} C^{\tau} \left(t_k+\frac{\tau}{3} \right)+ V_1 C^{\tau} \left(t_k+\frac{\tau}{3} \right) \right) \\
 \label{ii}
 \dfrac{C^{\tau} \left(t_k+\frac{2\tau}{3} \right) - C^{\tau}(t_k+\frac{\tau}{3})}{\tau} & = \partial_{x_2} \left(D \partial_{x_2} C^{\tau} \left(t_k+\frac{2\tau}{3} \right)+ V_2 C^{\tau} \left(t_k+\frac{2\tau}{3} \right) \right) \\ \label{iii}
\dfrac{C^{\tau} \left(t_{k+1}\right)-C^{\tau} \left(t_k+\frac{2\tau}{3} \right) }{\tau} & =
 \left( 
  \begin{array}{ccc}      
  -(k_0 + k_1)& k_3 & 0 \\         
   0 & -(k_0 +k_3) & k_2  \\        
   k_1 & 0 & -(k_0 +k_2) \\            
  \end{array}
\right)  C^{\tau}(t_{k+1}) 
\end{align}

In space we use a finite difference (or equivalently finite volume) method with stabilization to obtain robustness for the case of dominant convection we are interested in. The equations (\ref{i}) and (\ref{ii}) can be discretized with the Scharfetter-Gummel scheme (cf.{\cite{scharfetter}}), which is a variant of upwind schemes that has the advantage that all coefficients remain differentiable with respect to the parameters. For further details we refer to \cite{Reips}. Equation \eqref{iii} can be directly solved in each grid point.

The adjoint equations (see Appendix), which are a system of transport-reaction-diffusion equations as well, are discretization in an analogous manner, such that they are finally the discrete adjoint of the discretized forward model.

\section{Results}

In the following we present the results of two numerical tests on synthetic data, with the aim of estimating the potential of the nonlinear reconstruction approach. We use an operator $K$ as a given matrix (size $16512$ x $4225$) corresponding to a slice in a real PET scanner. A spatial discretization into $65 \times 65 = 4225$ pixels is chosen to match the operator resolution. The time step is chosen as $\tau = 3\cdot 10^{-5}$. Data are generated from forward simulations of the PDE system on fine grid with subsequent generation of Poisson noise.  

For the initisl radioactive concentration $C_{\mathcal{A}}$ in the artery we use the initial value
\begin{equation}\label{CAanfang}
C^0(x)= \tau(1 - x_1^2)(N - x_2)x_2
\end{equation}
with $N = 50$. 

A simple first test is to consider constant  parameters $k_1$, $k_2$ and $k_3$ (with a small random variation) in all pixels of the image, and their values of reconstruction are shown in the Table \ref{erstereconstk}. 

\begin{table}[H]
  \begin{center}
		\begin{tabular}{|c|c|} 
		 \hline
		  Parameter & Mean value of the reconstruction \\  
		 \hline
      $k_1 $ & 0.8263 $\pm 10^{-10}$ (1/cm)\\
      \hline
      $k_2 $ & 0.6886 $\pm 10^{-11}$ (1/cm)\\
      \hline
      $k_3 $ & 0.8264 $\pm 10^{-11}$ (1/cm)\\
      \hline
   \end{tabular}
   \caption{Reconstruction of $k_1$, $k_2$ and $k_3$}
   \label{erstereconstk}
  \end{center}
\end{table}
\FloatBarrier

In the following we report the results of two test with local defects in perfusion, the most relevant clinical case. To simulate such regions, we locally set the parameters $k_1$ and $k_2$ to zero
The table below shows the input values for the simulation. The first column refers to the biological parameter to be reconstructed and the second column contains the adopted value for each parameter at the beginning of numerical simulation. For further comparisons we refer to \cite{Levick, Wernick}, from where we deduced realistic values for all parameters.

\begin{table}[H] 
\begin{center}
		\begin{tabular}{|c|c|} 
		\hline
		 Parameter & Initial Value \\  
		 \hline
      $k_1 (*) (1/cm)$ & 0.9 (0)\\
      \hline
      $k_2 (*) (1/cm)$ & 0.75 (0)\\
      \hline
      $k_3 (1/cm)$ & 0.9 \\
      \hline
      $V_{x_{\mathcal{A}}} (cm/s)$ & 0.0001 \\
      \hline
      $V_{y_{\mathcal{A}}} (cm/s)$ & 700 \\
      \hline
      $V_{x_{\mathcal{T}}} (cm/s)$ & -50 \\
      \hline
      $V_{y_{\mathcal{T}}} (cm/s)$ & 0.0001 \\
      \hline
      $V_{x_{\mathcal{V}}} (cm/s)$ & 0.0001 \\
      \hline
      $V_{y_{\mathcal{V}}} (cm/s)$ & 700 \\
      \hline
      $D_{\mathcal{A}} (cm^2/s)$ & $3\cdot 10^{(-7)}$ \\
      \hline
      $D_{\mathcal{T}} (cm^2/s)$ & $3 \cdot 10^{(-6)}$ \\
      \hline
      $D_{\mathcal{V}} (cm^2/s)$ & $3 \cdot 10^{(-7)}$ \\
      \hline     
		\end{tabular}
				      \caption{Test data for reconstruction experiments}
		      \label{tabelaparanfangrealdata}
					\end{center}
		\end{table}
Here we also evaluate the behavior of radioactive flow on the region where $k_1$ and $k_2$ are equal to zero. Thus, in the above table, the symbol $(*)$ refers to the fact that $k_1$ and $k_2$ are not considered constant across the region of interest. When $k_1 = k_2 = 0$ there is no exchange of materials from the artery to the tissue and from the tissue to the vein, and this means that the radioactive concentration (in this region) in the tissue and in 
the vein are zero.

\subsection{Example 1: Small Defects in Perfusion}

We start with the case of a small defect in a thin region close to the left (inflow) boundary and present the corresponding reconstruction of parameters. Since the reaction rates $k_i$ are the relevant ones for medical issues, we focus on them in the following. 
The reconstruction of $k_3$ is almost constant (therefore the figure is omitted) with value $0.8061 \pm 10^{-9} /cm$. The following figures refer to the reconstruction of biological parameters for real PET-data: \\

\vspace{2.8cm}
\begin{center}
\begin{figure} [htb!]
\begin{minipage}{0.55\textwidth}
\hspace{0.7cm}
\includegraphics[scale=0.4]{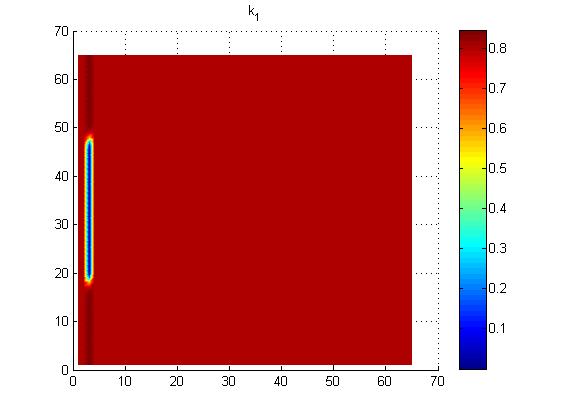}
\caption{Reconstruction of $k_1$}
\end{minipage}
\begin{minipage}[t]{0.45\textwidth}
\includegraphics[scale=0.4]{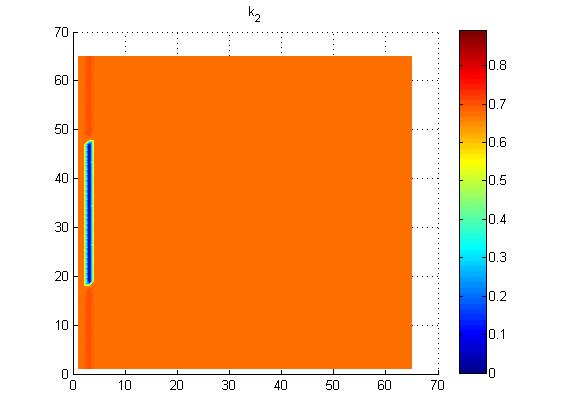}
\caption{Reconstruction of $k_2$}
\end{minipage}	
\end{figure}
\end{center}
\FloatBarrier



\subsection{Second Example: Inner dead region}


In the second case we consider a defect in smaller region in the interior.
For the radioactive concentration $C_{\mathcal{A}}$ in the artery we use the initial function given by the equation (\ref{CAanfang})  and use the time step $\tau = 3\cdot 10^{-5}$ in domain $\Omega$.  Again, the value of the reconstruction of $k_3$ is almost constant (therefore the figure is omitted) with value $0.0106 \pm 10^{-8} 1/cm$. 
%

The following figures refer to the reconstruction of biological parameters for realistic PET-data: 

\vspace{3.2cm}
\begin{center}
\begin{figure} [htb]
\begin{minipage}[t]{0.55\textwidth}
\hspace{0.8cm}
\includegraphics[scale=0.26]{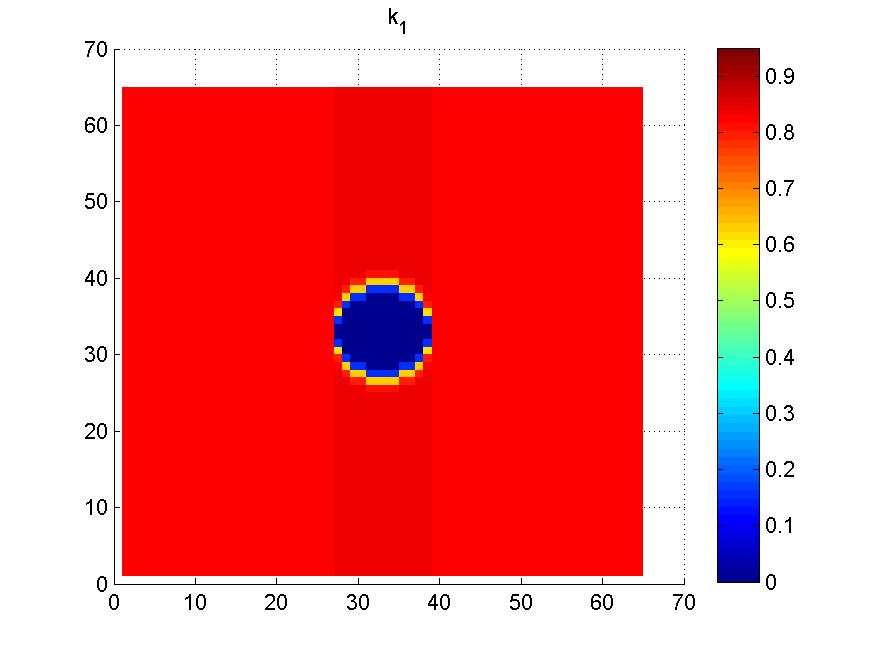}
\caption{Reconstruction of $k_1$}
\end{minipage}
\begin{minipage}[t]{0.45\textwidth}
\includegraphics[scale=0.26]{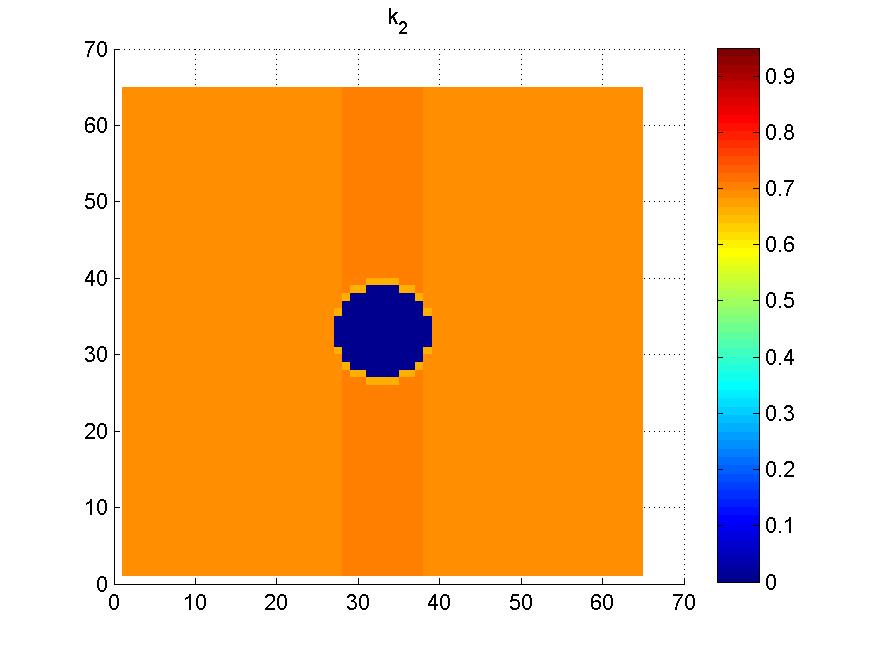}
\caption{Reconstruction of $k_2$}
\end{minipage}		
\end{figure}	
\end{center}
%
%
As we can see, the fact that $k_1$ and $k_2$ are equal to zero in the center is reflected in the graphics that represent the radioactive concentrations in tissue and vein, which remains zero in the same place. The plots of velocities below confirms the idea that sensitivity of data with respect to those is lower and the reconstruction is more difficult. However, the errors made in the velocities seems not to affect the the reconstruction of the defect regions too strongly.

\vspace{3.5cm}
\begin{center}
\begin{figure}[htb]
\begin{minipage}[t]{0.55\textwidth}
\hspace{0.8cm}
\includegraphics[scale=0.26]{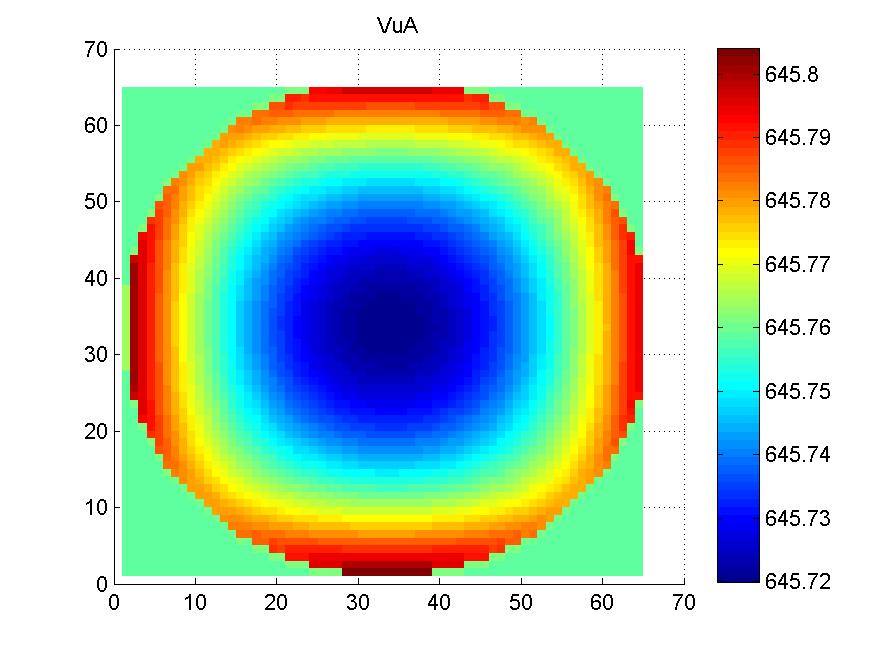}
\caption{Reconstruction of $V_{y_{\mathcal{A}}} $}
\end{minipage}
\begin{minipage}[t]{0.55\textwidth}
\includegraphics[scale=0.26]{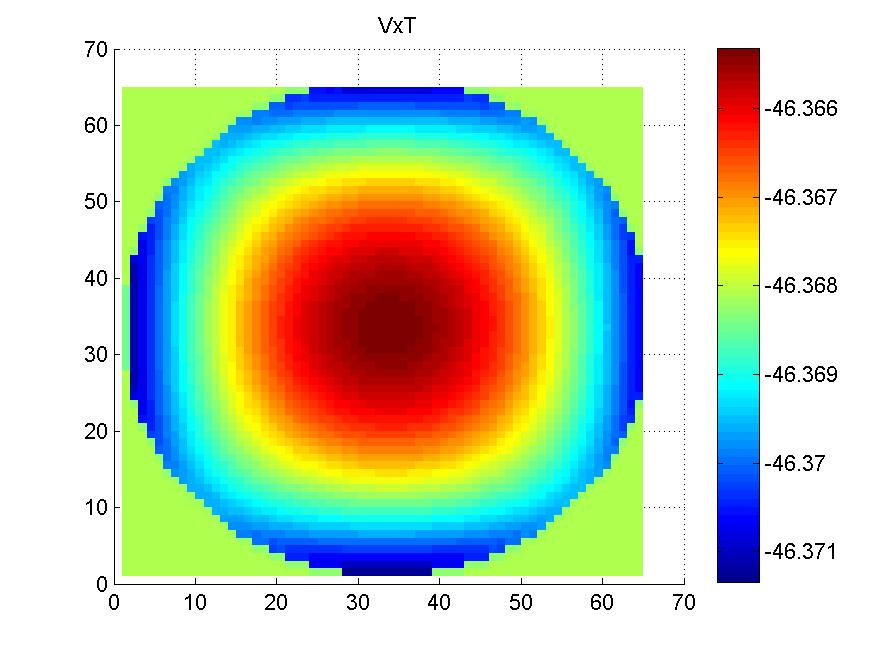}
\caption{Reconstruction of $V_{x_{\mathcal{T}}} $}
\end{minipage}
\end{figure}
\end{center}
\FloatBarrier
\begin{center}
\begin{figure}[H]
\hspace{4.5cm}
\includegraphics[scale=0.26]{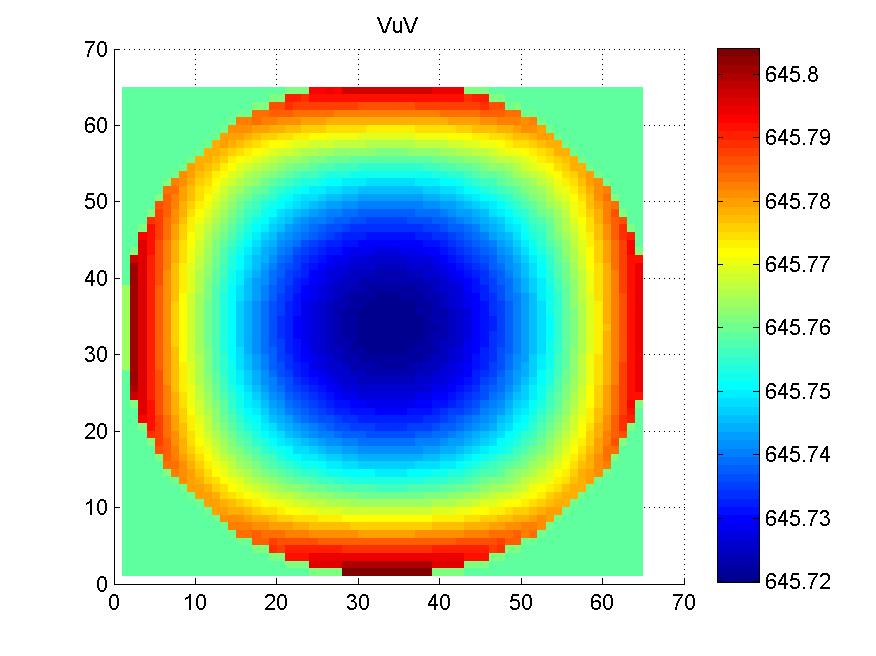}
\caption{Reconstruction of $V_{y_{\mathcal{V}}} $}
\end{figure}
\end{center}
\FloatBarrier


\section{Conclusion} 

In this paper we have introduced a novel approach for quantitative PET, which is capable of computing parameter reconstructions in presence of flow conditions. We have derived a novel model and a detailed analysis indicating the feasibility, which is confirmed by first computational tests. 

At least the following two issues are highly relevant for further research: From the practical point of view the efficient computational realization in three spatial dimensions and realistic PET setups. From a theoretical point of view it is highly relevant to understand conditions for unique identifiability of all parameters, first discussions in this direction can be found in \cite{Reips}.

\section*{Acknowledgements}

This work was carried out when LR was with the Institute for Computational and Applied Mathematics, WWU M\"unster. MB and RE acknowledge partial support by the German Science Foundation (DFG) via SFB 656, Subproject B2, and Cells-in-Motion Cluster of Excellence (EXC 1003 - CiM), WWU M\"unster, Germany. 


\bibliographystyle{plain}
\bibliography{referenc}

\section*{Appendix}
\appendix{\textbf{Adjoint Equations}}

The purpose of this section is the development of the parameter identification problem to allow the calculation of all the biological parameters that composes the vector $p$. 
Thus, minimizing the function below (with the regularization added) we can find the values that correspond to the desired physiological parameters
\begin{align*}
	\frac{1}{2} \int\limits_{0}^{T}\int\limits_{\Omega} \frac{\left(u - u_{k + \frac{1}{2}}\right)^{2}}{u_{k}} dx dt + \mathcal{R}(p) + \int\limits_{0}^{T}\int\limits_{\Omega} \left(G(p) - u\right)q\, dx
dt \rightarrow \min_{p} \, \text{,}
	\label{eqn:minproblem4}
\end{align*}
with $G(p) = G(p(x,t)) = u(x,t)$, for all $(x,t) \in \Omega \times [0,T]$. With the associated Lagrange functional one has
\begin{align*}
\mathcal{L}(u, p;q) = \dfrac{1}{2} \int\limits_{0}^{T} \int\limits_{\Omega} \dfrac{\left(u - u_{k+\frac{1}{2}}\right)^2}{u_k} dx dt + \mathcal{R}(p) + \int\limits_{0}^{T} \int\limits_{\Omega} (G(p) - u)\ q\ dx dt
\end{align*}

\normalsize
One must now calculate the optimality conditions to the problem, which means that all the partial Fr\'echet-derivatives must be zero. Thus, we obtain
\vspace{0.2cm}
\begin{equation*}\label{partialu}
\begin{split}
\dfrac{\partial \mathcal{L}}{\partial u} = \dfrac{u(x,t) - u_{k + \frac{1}{2}} (x,t)}{u_k(x,t)} - q(x,t) = 0
\end{split}
\end{equation*}

The optimality conditions for $k_1(x)$, $k_2(x)$, $k_3(x)$, $V_{\mathcal{T}}(x)$, $V_{\mathcal{A}}(x)$, $V_{\mathcal{V}}(x)$, $D_{\mathcal{T}}(x)$, $D_{\mathcal{A}}(x)$ and $D_{\mathcal{V}}(x)$ are
\begin{equation*}\label{derivk1completo}
\begin{split}
\dfrac{\partial \mathcal{L}}{\partial k_1} & = \alpha \big( \Lambda_{\mathcal{T}}(x)(k_1(x) - k_1^*) + \Lambda_{\mathcal{A}}(x)(k_1(x) - k_1^*) \big) -  \xi\big ( \Lambda_{\mathcal{T}}(x) \Delta k_1(x) +  \Lambda_{\mathcal{A}}(x) \Delta k_1(x) \big)\\
& - \int\limits_{0}^{T}  C_\mathcal{A}(x,t) \mu(x,t)dt + \int\limits_{0}^{T} C_\mathcal{A}(x,t) \eta(x,t)dt \\
\end{split}
\end{equation*}
\begin{equation*}\label{derivk2completo}
\begin{split}
\dfrac{\partial \mathcal{L}}{\partial k_2} & = \alpha \big( \Lambda_{\mathcal{T}}(x)(k_2(x) - k_2^*) + \Lambda_{\mathcal{V}}(x)(k_2(x) - k_2^*) \big) -  \xi\big ( \Lambda_{\mathcal{T}}(x) \Delta k_2(x) +  \Lambda_{\mathcal{V}}(x) \Delta k_2(x) \big)\\
& + \int\limits_{0}^{T}  C_\mathcal{T}(x,t) \mu(x,t)dt - \int\limits_{0}^{T} C_\mathcal{T}(x,t) \gamma(x,t)dt \\
\end{split}
\end{equation*}
\begin{equation*}\label{derivk3completo}
\begin{split}
\dfrac{\partial \mathcal{L}}{\partial k_3} & = \alpha \big( \Lambda_{\mathcal{A}}(x)(k_3(x) - k_3^*) + \Lambda_{\mathcal{V}}(x)(k_3(x) - k_3^*) \big) -  \xi\big ( \Lambda_{\mathcal{A}}(x) \Delta k_3(x) +  \Lambda_{\mathcal{V}}(x) \Delta k_3(x) \big)\\
& - \int\limits_{0}^{T}  C_\mathcal{V}(x,t) \eta(x,t)dt + \int\limits_{0}^{T} C_\mathcal{V}(x,t) \gamma(x,t)dt \\
\end{split}
\end{equation*}
\begin{equation*}\label{partialVTcompleto}
\begin{split}
\dfrac{\partial \mathcal{L}}{\partial V_\mathcal{T}} & = \int\limits_{0}^{T}  V_\mathcal{T}(x) \cdot \nabla \mu(x,t)\hspace{0.1cm} dt + \alpha(V_{\mathcal{T}}(x) - V_{\mathcal{T}}^{*}) - \xi(\Lambda_{\mathcal{T}}(x) \Delta V_{\mathcal{T}}(x))\\
\end{split}
\end{equation*}

\begin{equation*}\label{partialVAcompleto}
\begin{split}
\dfrac{\partial \mathcal{L}}{\partial V_\mathcal{A}} & = \int\limits_{0}^{T}  V_\mathcal{A}(x) \cdot \nabla \eta(x,t) \hspace{0.1cm} dt + \alpha(V_{\mathcal{A}}(x) - V_{\mathcal{A}}^{*}) - \xi(\Lambda_{\mathcal{A}}(x) \Delta V_{\mathcal{A}}(x))\\
\end{split}
\end{equation*}
\begin{equation*}\label{partialVVcompleto}
\begin{split}
\dfrac{\partial \mathcal{L}}{\partial V_\mathcal{V}} & = \int\limits_{0}^{T}  V_\mathcal{V}(x) \cdot \nabla \gamma(x,t) \hspace{0.1cm} dt + \alpha(V_{\mathcal{V}}(x) - V_{\mathcal{V}}^{*}) - \xi(\Lambda_{\mathcal{V}}(x) \Delta V_{\mathcal{V}}(x))\\
\end{split}
\end{equation*}
\begin{equation*}\label{partialDTcompleto}
\begin{split}
\dfrac{\partial \mathcal{L}}{\partial D_\mathcal{T}} & = \int\limits_{0}^{T}  \nabla C_\mathcal{T}(x) \cdot \nabla \mu(x,t) \hspace{0.1cm} dt + \alpha(D_{\mathcal{T}}(x) - D_{\mathcal{T}}^{*}) - \xi(\Lambda_{\mathcal{T}}(x) \Delta D_{\mathcal{T}}(x)) \\
\end{split}
\end{equation*}
\begin{equation*}\label{partialDAcompleto}
\begin{split}
\dfrac{\partial \mathcal{L}}{\partial D_\mathcal{A}} & = \int\limits_{0}^{T}  \nabla C_\mathcal{A}(x) \cdot \nabla \eta(x,t)\hspace{0.1cm}  dt + \alpha(D_{\mathcal{A}}(x) - D_{\mathcal{A}}^{*}) - \xi(\Lambda_{\mathcal{A}}(x) \Delta D_{\mathcal{A}}(x)) \\
\end{split}
\end{equation*}
\begin{equation*}\label{partialDVcompleto}
\begin{split}
\dfrac{\partial \mathcal{L}}{\partial D_\mathcal{V}} & = \int\limits_{0}^{T}  \nabla C_\mathcal{V}(x) \cdot \nabla \gamma(x,t) \hspace{0.1cm}dt + \alpha(D_{\mathcal{V}}(x) - D_{\mathcal{V}}^{*}) - \xi(\Lambda_{\mathcal{V}}(x) \Delta D_{\mathcal{V}}(x)) \\
\end{split}
\end{equation*}
And we apply the Forward-Backward Splitting method for all parameters that composes the vector $p$ to obtain:

\begin{equation*}\label{k_1neu}
\begin{split}
k_1^{k+1}(x,y) & = (1 + 2 \alpha \tau - 2 \xi \tau B_x - 2\xi \tau B_y)^{-1} \\
& \left (k_1^{k}(x,y) + \tau \int\limits_{0}^{T} C_{\mathcal A}(x,y,t) \mu(x,y,t) dt - \tau \int\limits_{0}^{T} C_{\mathcal A}(x,y,t) \eta(x,y,t) dt + 2\alpha \tau k_1^{*} \right)\\
\end{split}
\end{equation*}
\begin{equation*}\label{k_2neu}
\begin{split}
k_2^{k+1}(x,y) & = (1 + 2 \alpha \tau - 2 \xi \tau B_x - 2\xi \tau B_y)^{-1} \\
& \left (k_2^{k}(x,y) - \tau \int\limits_{0}^{T} C_{\mathcal T}(x,y) \mu(x,y) dt + \tau \int\limits_{0}^{T} C_{\mathcal T}(x,y,t) \gamma(x,y,t) dt + 2\alpha \tau k_2^{*} \right)\\
\end{split}
\end{equation*}
\begin{equation*}\label{k_3neu}
\begin{split}
k_3^{k+1}(x,y) & = (1 + 2 \alpha \tau - 2 \xi \tau B_x - 2\xi \tau B_y)^{-1}\\
& \left (k_3^{k}(x,y,t) + \tau(x,y,t) \int\limits_{0}^{T} C_{\mathcal V}(x,y,t) \eta(x,y,t) dt + \tau \int\limits_{0}^{T} C_{\mathcal V}(x,y,t) \gamma(x,y,t) dt + 2\alpha \tau k_3^{*} \right)\\
\end{split}
\end{equation*}
\begin{equation*}\label{VTneu}
V_{\mathcal{T}}^{k+1}(x,y) = (1 -  \alpha \tau + \xi \tau B_x + \xi \tau B_y)^{-1} \left (V_{\mathcal{T}}^{k}(x,y) - \tau V_{\mathcal{T}}^{k}(x,y) \cdot \nabla \int\limits_{0}^{T} \mu (x,y,t) dt + \alpha \tau V_{\mathcal{T}}^{*} \right)
\end{equation*}
\begin{equation*}\label{VAneu}
V_{\mathcal{A}}^{k+1}(x,y) = (1 -  \alpha \tau + \xi \tau B_x + \xi \tau B_y)^{-1} \left (V_{\mathcal{A}}^{k}(x,y) - \tau V_{\mathcal{A}}^{k}(x,y,t) \cdot \nabla \int\limits_{0}^{T} \eta(x,y,t) dt + \alpha \tau V_{\mathcal{A}}^{*} \right)
\end{equation*}
\begin{equation*}\label{VVneu}
V_{\mathcal{V}}^{k+1}(x,y) = (1 -  \alpha \tau + \xi \tau B_x + \xi \tau B_y)^{-1} \left (V_{\mathcal{V}}^{k}(x,y) - \tau V_{\mathcal{V}}^{k}(x,y) \cdot \nabla \int\limits_{0}^{T} \gamma(x,y,t) dt + \alpha \tau V_{\mathcal{V}}^{*} \right)
\end{equation*}
\begin{equation*}\label{DTneu}
\begin{split}
D_{\mathcal{T}}^{k+1}(x,y) & = (1 + \alpha \tau - \xi \tau B_x - \xi \tau B_y)^{-1} \\
& \left (D_{\mathcal{T}}^{k}(x,y) - \tau (\nabla D_{\mathcal{T}}^{k}(x,y) \int\limits_{0}^{T} \nabla \mu(x,y,t) dt )+ \alpha \tau D_{\mathcal{T}}^{*} \right)\\
\end{split}
\end{equation*}
\begin{equation*}\label{DAneu}
\begin{split}
D_{\mathcal{A}}^{k+1}(x,y) & = (1 + \alpha \tau - \xi \tau B_x - \xi \tau B_y)^{-1} \\
& \left (D_{\mathcal{A}}^{k}(x,y) - \tau (\nabla D_{\mathcal{A}}^{k}(x,y) \int\limits_{0}^{T} \nabla \eta(x,y,t) dt )+ \alpha \tau D_{\mathcal{A}}^{*} \right)\\
\end{split}
\end{equation*}
\begin{equation*}\label{DVneu}
\begin{split}
D_{\mathcal{V}}^{k+1}(x,y) & = (1 + \alpha \tau - \xi \tau B_x - \xi \tau B_y)^{-1} \\
& \left (D_{\mathcal{V}}^{k}(x,y) - \tau (\nabla D_{\mathcal{V}}^{k}(x,y) \int\limits_{0}^{T} \nabla \gamma(x,y,t) dt )+ \alpha \tau D_{\mathcal{V}}^{*} \right)\\
\end{split}
\end{equation*}

A good choice of $\tau$ defines a significant speedup, because the dependence on the ill-posedness of the operator $K$ (the ill-conditioning of the matrix that represents the discretization of $K$)
can make the iterative scheme very slow. \\

\appendix{\textbf{Example 1: Small Defects in Perfusion - Regularization Parameters}}

The following table shows all the regularization parameters for the example 1:

\begin{table}[H] 
\begin{center}
		\begin{tabular}{|c|c|c|c|} 
		\hline
		 Parameter & $(\cdot)^*$& A-p. Regularization ($\alpha$) & Gradient regularization ($\xi$) \\  
		 \hline
      $k_1 $ &  0.89 & 0.0171 & 0.0008 \\
      \hline
      $k_2 $ &  0.7 & 0.0158 & 0.0001 \\
      \hline
      $k_3 $ & 0.85 & 0.0164 & 0.0001 \\
      \hline
      $V_{x_{\mathcal{A}}} $ & 0.1 & 0.0010 & 0.0001 \\
      \hline
      $V_{y_{\mathcal{A}}} $ & 15 & 1.1000 & 0.0001 \\
      \hline
      $V_{x_{\mathcal{T}}} $ & -5 & 1.1220 & 0.0001 \\
      \hline
      $V_{y_{\mathcal{T}}} $ & 0.1 & 0.0010 & 0.0001 \\
      \hline
      $V_{x_{\mathcal{V}}} $ & 0.1 & 0.0010 & 0.0001 \\
      \hline
      $V_{y_{\mathcal{V}}} $ & 15 & 1.1000 & 0.0001 \\
      \hline
      $D_{\mathcal{A}} $ & $10^{(-3)}$ & 0.0003 & 0.0004 \\
      \hline
      $D_{\mathcal{T}} $ & $10^{(-2)}$ & 0.0003 & 0.0004 \\
      \hline
      $D_{\mathcal{V}} $ & $10^{(-3)}$ & 0.0003 & 0.0004 \\
      \hline     
		\end{tabular}
		      \caption{Input regularization parameters for a first real example}
		      \label{tabelaparanfangregularizationdata}
					\end{center}
		\end{table}

The $(\cdot)^*$ refers to a-priori knowledge in the regularization functional for each parame-\\ ter of the problem. Whereas, for example, the velocity of the radioactive concentration in the artery has a typical value of $V_{\mathcal{A}}^{*}$, we can regularize $V_{\mathcal{A}}$ by
\begin{equation*} \label{regaprioriestrela}
\mathcal{R}(V_{\mathcal{A}}(x)) = \dfrac{\alpha}{2} \int\limits_{\Omega} (V_{\mathcal{A}} - V_{\mathcal{A}}^{*})^2 dx
\end{equation*}
%
where $\alpha$ (values shown in the third column) denotes the regularization parameter, $\alpha \in \mathbb{R}_{+}$. 

Like the a-priori regularization we apply the Gradient regularization in each parameter independently. The regularization of the gradient is designed to ensure (guarantee) smoothness in space and time, adding a bound to the spatial gradients ($\nabla k_1, \nabla k_2, \nabla k_3, \\ \nabla V_{\mathcal{A}}, \nabla V_{\mathcal{T}}, \nabla V_{\mathcal{V}}, \nabla D_{\mathcal{A}}, \nabla D_{\mathcal{T}}, \nabla D_{\mathcal{V}}$). The regularization added to the terms is given by
\begin{equation*}
\mathcal{R}_{\xi, \Phi} (g) = \dfrac{\xi}{2} \int\limits_{\Phi} |\nabla g(x)|^2 dx
\end{equation*}
with $\Phi \in \Omega$. Thus, the fourth column refers to the terms $\xi$ for each biological parameter in the above equation.
%


\end{document}